\documentclass[11pt]{article}
\usepackage{amsmath,amscd}
\usepackage{amssymb,latexsym,amsthm}
\usepackage{color,palatino}
\usepackage[utf8]{inputenc}
\usepackage[T1]{fontenc}
\usepackage{csquotes}
\usepackage{latexsym}
\usepackage{multirow}
\usepackage{lscape}
\usepackage{enumerate}
\usepackage{longtable}
\usepackage{array}
\usepackage{amsmath}
\usepackage{amssymb}
\usepackage{fancyhdr}
\usepackage{pb-diagram}
\usepackage{amsfonts}
\usepackage{hyperref}
\usepackage{url}
\hypersetup{
  colorlinks=true,
  linkcolor=blue,
  citecolor=black,
  urlcolor=cyan
}
\usepackage{pb-diagram}
\newtheorem{theorem}{Theorem}[section]
\newtheorem{proposition}[theorem]{Proposition}
\newtheorem{lemma}[theorem]{Lemma}

\newtheorem{corollary}[theorem]{Corollary}

\newtheorem{definition}[theorem]{Definition}

\newtheorem*{question*}{Question}
\theoremstyle{definition}
\newtheorem{remark}[theorem]{Remark}

\newcommand{\K}{\Bbbk}

\DeclareMathOperator{\GKdim}{GKdim}

\makeatletter
\def\@roman#1{\romannumeral #1}
\makeatother

\title{\textbf{The ozone group of $U_q^+(B_2)$}}
\author{James Gómez and Helbert Venegas}

\date{}
\begin{document}
\makeatletter
\def\@roman#1{\romannumeral #1}
\makeatother
\maketitle
\begin{abstract}
\noindent
We determine the ozone group of \(U_q^{+}(B_2)\). More precisely, if
\(\ell=\operatorname{ord}(q^2)\) when \(q\) is a root of unity, we obtain
$$
\operatorname{Oz}(U_q^{+}(B_2))
\cong
\begin{cases}
\mu_2 , & \text{if \(q\) is not a root of unity},\\[2mm]
\mu_{\gcd(\ell,2)}, &
\text{if \(q\) is a primitive \(m\)-th root of unity, \(m\geq5\).}
\end{cases}
$$
We also study several homological properties of \(U_q^{+}(B_2)\), showing
that it is Artin--Schelter regular of global dimension \(4\), Auslander-regular, 
Cohen--Macaulay, strongly Noetherian, and skew Calabi--Yau. Using two different iterated Ore presentations, we show that its Nakayama
automorphism is trivial for every $q\in\K^*$, so $U_q^{+}(B_2)$ is
Calabi--Yau of dimension $4$; for $q$ not a root of unity this also follows
from work of Goodearl and Yakimov. When $q$ is a primitive $m$-th root of unity with $m\geq5$, we deduce that
the center is Gorenstein, we check directly that the ozone group acts with
trivial homological determinant, and we describe the nonzero normal elements
in terms of the ozone group.
 
\bigskip

\bigskip

\end{abstract}
\section{Introduction}
The algebra \(U_q^{+}(B_2)\)  was originally defined within the general framework of quantum enveloping algebras developed by Drinfeld \cite{Drinfeld1986} and Jimbo \cite{Jimbo1985} in the 1980s, as part of the theory of quantum groups arising from the study of quantum integrable systems and the Yang-Baxter equation.
Its significance lies in the fact that it constitutes a fundamental example of a rank-$2$ quantum algebra (with two simple roots), associated to the complex simple Lie algebra $\mathfrak{g}$  of type $B_2$. In particular $U_q^{+}(B_2)$ is the positive part of the quantum group $U_q(\mathfrak{g})$, generated by the Chevalley generators $E_1$, $E_2$ corresponding to the simple roots of the root system of type $B_2$. For the case where \(q\) is not a root of unity, this algebra has been studied in detail by Launois \cite{Launois2007}, who explicitly described its prime and primitive spectra, computed its automorphism group, and obtained quantum analogues of the Weyl algebra as simple quotients. Additionally, in \cite{AndruskiewitschDumas2008} some results related to Hopf algebra automorphism are also developed. 
Recently, some properties have been studied in the case where the parameter $q$ is a root of unity. In particular, Bera and Mukherjee \cite{BeraMukherjee2026UqB2} proved that $U_q^{+}(B_2)$ is a PI algebra when $q$ is a primitive $m$-th root of unity with $m\geq 5$, computed its center, and classified its simple modules up to isomorphism. \\

Throughout this paper, \(\K\) denotes an algebraically closed field of
characteristic zero, and all algebras are assumed to be
\(\K\)-algebras.\\

\noindent Formally, the algebra   \(U_q^{+}(B_2)\) is the \(\K\)-algebra generated by two indeterminates  $e_1$ and $e_2$ subject to the quantum Serre relations,
\begin{align*}
e_1^{2}e_2 - (q^{2}+q^{-2})\,e_1 e_2 e_1 + e_2 e_1^{2} &= 0\\
e_2^{3}e_1 - (q^{2}+1+q^{-2})\,e_2^{2} e_1 e_2 + (q^{2}+1+q^{-2})\,e_2 e_1 e_2^{2}   - e_1 e_2^{3} &= 0
\end{align*}

Andruskiewitsch and Dumas \cite[Section 3.1.2]{AndruskiewitschDumas2008},   introduced a new set of generators for
$U_q^{+}(B_2)$. For this, they defined
$e_{3}=e_{1}e_{2}-q^{2}e_{2}e_{1}
$ and $
z=e_{2}e_{3}-q^{2}e_{3}e_{2}$. Thus, the set 
\(\{\, z^{i} e_{3}^{j} e_{1}^{k} e_{2}^{\ell} \mid i,j,k,\ell \in \mathbb{N}\,\}\) 
is a PBW basis for $U_q^{+}(B_2)$  and the latter turns out to be  the \(\K\)-algebra generated by
\(e_1,e_2,e_3\) and \(z\) subject to the relations:
\begin{align*}
&e_i z = z e_i, \ \text{ for }\, i=1,2,3,
&e_1 e_3 &= q^{-2} e_3 e_1,\\
&e_2 e_1 = q^{-2} e_1 e_2 - q^{-2} e_3, &e_2 e_3 &= q^{2} e_3 e_2 + z.
\end{align*}
Furthermore, \(U_q^{+}(B_2)\) can be presented as an iterated Ore extension, 
\[
U_q^{+}(B_2) \;=\; \K[z,e_3][\,e_1;\sigma_1\,][\,e_2;\sigma_2,\delta_2\,],
\]
where
\begin{align*}
\sigma_1(z) &= z, 
& \sigma_1(e_3) &= q^{-2}e_3, \\[6pt]
\sigma_2(z) &= z, 
& \sigma_2(e_3) &= q^{2}e_3, 
& \sigma_2(e_1) &= q^{-2}e_1, \\[6pt]
\delta_2(z) &= 0, 
& \delta_2(e_3) &= z, 
& \delta_2(e_1) &= -q^{-2}e_3.
\end{align*}
In 2024, Kenneth Chan, Jason Gaddis, Robert Won, and James J. Zhang introduced the notion of the \textbf{ozone group} of an algebra \(A\) (see \cite{ChanGaddisWonZhang2024OzoneGroupsCenters}), denoted by \(\operatorname{Oz}(A)\), and defined by
$$
\operatorname{Oz}(A):=\operatorname{Aut}_{Z\text{-alg}}(A).
$$
\noindent Thus, \(\operatorname{Oz}(A)\) consists of those automorphisms of \(A\) that fix its center pointwise.  If \(\sigma\in\operatorname{Oz}(A)\), then \(\sigma\) acts as the identity on the center \(Z(A)\), that is,
\(\left.\sigma\right|_{Z(A)}=\operatorname{id}_{Z(A)}\). \\

The study of ozone groups of PI Artin--Schelter regular algebras has seen significant recent development. Chan, Gaddis, Won, and Zhang
\cite{ChanGaddisWonZhang2025} established general results concerning their
finiteness, their relationship with normal elements, the rank over the center,
and connections with the Calabi--Yau and Gorenstein properties. Subsequently,
Liu, Wu, and Zhu \cite{LiuWuZhu2026} proved that the ozone group of every
noetherian PI Artin--Schelter regular algebra is abelian. In this direction,
Gaddis and Yee \cite{GaddisYee2026} recently studied the family
\(B_q(f)\), obtaining families of Calabi--Yau algebras with trivial ozone
group, while the authors \cite{GomezVenegas2026} considered the
ozone groups of this family and showed that, when \(q\) has order \(n>1\),
\(\operatorname{Oz}(B_q(f))\cong\mu_e\times\mu_e\), where
\(e=\gcd(n,\{j+1:j\in\operatorname{supp}(f)\})\). \\

For \(U_q^{+}(B_2)\), although its PI property, PI degree, and center at
roots of unity have recently been studied
\cite{BeraMukherjee2026UqB2}, an explicit computation of its ozone group
does not appear to have been previously recorded. In this work, we determine the ozone group of $U_q^{+}(B_2)$ in
both the generic and the root-of-unity cases. In the latter case,
if $q$ is a primitive $m$-th root of unity with $m\geq5$ and
$\ell=\operatorname{ord}(q^2)$, we obtain
$\operatorname{Oz}(U_q^{+}(B_2))\cong\mu_{\gcd(\ell,2)}$.
We also establish that $U_q^{+}(B_2)$ is Artin--Schelter regular
of global dimension $4$, Auslander--regular, Cohen--Macaulay,
and strongly Noetherian by realizing it as a graded bijective
skew PBW extension. Finally, following a suggestion of Xin Tang, two different iterated Ore
presentations allow us to compute its Nakayama automorphism and show that
$U_q^{+}(B_2)$ is Calabi--Yau of dimension $4$ for every $q\in\K^*$. This is
an instance of the method of Liu, Wang and Wu \cite{LiuWangWu2014Ore}, as
formalized by Goodearl and Yakimov \cite{GoodearlYakimov2015}; for $q$ not a
root of unity the result was already known \cite{LiuWu2014,GoodearlYakimov2015}.
At roots of unity it gives the Gorenstein property of the center for every
$m\geq5$, not only when $\ell$ is odd.

\section{Preliminaries}
In this section, we introduce the main notions and homological properties that will be used throughout the paper.\\

We recall that a \(\K\)-algebra \(A\) is called a \emph{PI algebra} if
there exists a nonzero polynomial
\(f\in\K\langle x_1,\ldots,x_n\rangle\) such that
\(f(a_1,\ldots,a_n)=0\) for all \(a_1,\ldots,a_n\in A\)
\cite[\S13.1]{mcconnell}. Such a polynomial is called a
\emph{polynomial identity} of \(A\).

We also recall that a connected graded algebra
\(A=\bigoplus_{n\geq0}A_n\), with \(A_0=\K\), is called
\emph{Artin--Schelter Gorenstein} of dimension \(d\) if it has finite
injective dimension \(d\) on both sides and there exists
\(\mathfrak l\in\mathbb Z\) such that
\[
\operatorname{Ext}_{A}^{i}({}_A\K,{}_AA)
\cong
\operatorname{Ext}_{A^{\mathrm{op}}}^{i}(\K_A,A_A)
\cong
\begin{cases}
0, & i\neq d,\\
\K(\mathfrak l), & i=d.
\end{cases}
\]
If, in addition, \(A\) has finite global dimension and finite
Gelfand--Kirillov dimension, then \(A\) is called
\emph{Artin--Schelter regular}
\cite[Definition~0.1]{ChanGaddisWonZhang2025}.

For an \(A\)-module \(M\), the \emph{homological grade} is defined by
\[
j_A(M)=\min\{i:\operatorname{Ext}_A^i(M,A)\neq0\},
\]
with \(j_A(M)=\infty\) if no such \(i\) exists. 

A Noetherian algebra \(A\) is said to satisfy the
\emph{Auslander condition} if, for every finitely generated left or
right \(A\)-module \(M\), every \(i\geq0\), and every submodule
\(N\) of \(\operatorname{Ext}_A^i(M,A)\), one has
\[
j_A(N)\geq i.
\]
The algebra \(A\) is called \emph{Auslander--regular} if it satisfies
the Auslander condition and has finite global dimension
\cite[Definition~2.1]{LezamaVenegas2017Homological}. Moreover, \(A\) is
called \emph{Cohen--Macaulay} with respect to the classical
Gelfand--Kirillov dimension if
\[
\operatorname{GKdim}(A)
=
j_A(M)+\operatorname{GKdim}(M)
\]
for every nonzero Noetherian \(A\)-module \(M\)
\cite[Definition~3.1]{LezamaVenegas2017Homological}.

 Moreover, a
\(\K\)-algebra \(A\) is called \emph{strongly Noetherian} if
\(A\otimes_{\K}C\) is Noetherian for every commutative Noetherian
\(\K\)-algebra \(C\).

Finally, let \(A\) be a graded algebra and let
\(A^e=A\otimes_{\K}A^{\operatorname{op}}\). The algebra \(A\) is called
\emph{skew Calabi--Yau} of dimension \(d\) if it is homologically smooth
and there exist an automorphism \(\mu\) of \(A\) and
\(\ell\in\mathbb Z\) such that
\[
\operatorname{Ext}_{A^e}^{i}(A,A^e)
\cong
\begin{cases}
0, & i\neq d,\\
{}^{1}\!A^{\mu}(\ell), & i=d,
\end{cases}
\]
as graded \(A^e\)-modules
\cite[Definition~0.1]{ReyesRogalskiZhang2014}.
A skew Calabi--Yau algebra is called \emph{Calabi--Yau} if its Nakayama
automorphism is inner. In the connected graded domain setting considered
here, this is equivalent to the Nakayama automorphism being the identity.

\section[The case where q is not a root of unity]{The ozone group }
The ozone group of \(U_q^{+}(B_2)\) is determined by distinguishing the
cases where \(q\) is and is not a root of unity. In the generic case, we
use the known descriptions of the automorphism group and the center. In
the root-of-unity case, with \(\ell=\operatorname{ord}(q^2)\), the
structure of the center allows us to restrict the automorphisms that fix
it pointwise and thereby determine the ozone group.

\subsection[The case where q is not a root of unity]{When \(q\) is not a root of unity}
Suppose that \(q\in\K^{*}\) is not a root of unity. In this case, the
automorphism group of \(U_q^{+}(B_2)\) is known. Andruskiewitsch and Dumas
\cite[Remark~3.2 and Proposition~3.3]{AndruskiewitschDumas2008}
observed that, in type \(B_2\), the Dynkin diagram automorphism group is
trivial and determined the subgroup of automorphisms stabilizing the ideal
generated by the central element \(z\). Subsequently, Launois
\cite[Proposition~4.3 and Theorem~4.4]{Launois2007} proved that every
automorphism of \(U_q^{+}(B_2)\) stabilizes this ideal. Consequently,
\[
\operatorname{Aut}_{\K}\!\left(U_q^{+}(B_2)\right)
=
\left\{
\psi_{\alpha,\beta}
\mid
\alpha,\beta\in\K^{*}
\right\}
\cong
(\K^{*})^{2},
\]
where
\[
\psi_{\alpha,\beta}(e_1)=\alpha e_1,\qquad
\psi_{\alpha,\beta}(e_2)=\beta e_2,\qquad
\psi_{\alpha,\beta}(e_3)=\alpha\beta\,e_3,\qquad
\psi_{\alpha,\beta}(z)=\alpha\beta^2\,z.
\]
Andruskiewitsch and Dumas
\cite[Lemma~3.1]{AndruskiewitschDumas2008} proved that
\(Z\!\left(U_q^{+}(B_2)\right)=\K[z,z']\), where
\[
z'=(1-q^{-4})(1-q^{-2})z_1,
\qquad
z_1=e_1e_2e_3+\frac{e_1z}{q^2-1}+\frac{e_3^2}{q^4-1}.
\]
Therefore,
\[
Z\!\left(U_q^{+}(B_2)\right)=\K[z,z_1].
\]
\begin{theorem}
Suppose that \(q\in\K^{*}\) is not a root of unity. Then
\[
\operatorname{Oz}\!\left(U_q^{+}(B_2)\right)
=
\{\operatorname{id},\tau\},
\]
where
\[
\tau(e_1)=e_1,
\qquad
\tau(e_2)=-e_2,
\qquad
\tau(e_3)=-e_3,
\qquad
\tau(z)=z.
\]
\end{theorem}

\begin{proof}
For
\(\psi_{\alpha,\beta}\in\operatorname{Aut}(U_q^{+}(B_2))\),
we have
\[
\psi_{\alpha,\beta}(z)=\alpha\beta^2z,
\qquad
\psi_{\alpha,\beta}(z_1)=\alpha^2\beta^2z_1.
\]
It follows that
\(\psi_{\alpha,\beta}\in\operatorname{Oz}\!\left(U_q^{+}(B_2)\right)\)
if and only if
\[
\alpha\beta^2=1,
\qquad
\alpha^2\beta^2=1.
\]
These equalities imply
\[
\alpha=1,
\qquad
\beta^2=1.
\]
Since \(\operatorname{char}(\K)=0\), we obtain \(\beta=\pm1\). Therefore,
\[
\operatorname{Oz}\!\left(U_q^{+}(B_2)\right)
=
\{\psi_{1,1},\psi_{1,-1}\}
=
\{\operatorname{id},\tau\}.
\]
\end{proof}
\subsection[The case where q is a root of unity]{When \(q\) is a root of unity}
Throughout this subsection, we assume that \(q\in\K^*\) is a primitive \(m\)-th root of unity, with \(m\geq5\). Let
$$
\ell=\operatorname{ord}(q^2)
=
\frac{m}{\gcd(m,2)}.
$$
In particular,
$$
\ell=
\begin{cases}
m, & \text{if \(m\) is odd},\\
m/2, & \text{if \(m\) is even}.
\end{cases}
$$
\\
For \(k\geq 0\), define the \(q\)-numbers
\[
[k]_{q^2}:=\frac{q^{2k}-1}{q^2-1},
\qquad
[k]_{q^{-4}}:=\frac{q^{-4k}-1}{q^{-4}-1}.
\]
which are well defined because \(m\geq5\) forces \(q^{2}\neq1\) and
\(q^{4}\neq1\). These \(q\)-numbers satisfy the recurrence relations
\[
q^{2k}+[k]_{q^2}=[k+1]_{q^2},
\qquad
q^{-4k}+[k]_{q^{-4}}=[k+1]_{q^{-4}},
\]
which will be used throughout the sequel.

\begin{lemma}{\cite[Lemma 2.1]{BeraMukherjee2026UqB2}} \label{lem:B2}
The identities below hold in  $U_q^+(B_2)$:
\begin{enumerate}[\rm (i)]
\item $e_2 e_3^{\,k}=q^{2k} e_3^{\,k} e_2+ [k]_{q^2}\, z\, e_3^{\,k-1}.$
  \item $e_2^{\,k} e_3=q^{2k} e_3 e_2^{\,k}+[k]_{q^2}\, z\, e_2^{\,k-1}.$
  \item $e_2 e_1^{\,k}=q^{-2k} e_1^{\,k} e_2-q^{-2}\,[k]_{q^{-4}}\, e_3 e_1^{\,k-1}.$
  \item $e_2^{k} e_1= q^{-2k} e_1 e_2^{k}- q^{-2} B_k\, e_3 e_2^{k-1}
- C_k\, z\, e_2^{k-2}$, 
where 
\begin{align*}
B_k=\frac{q^{2k}-q^{-2k}}{q^{2}-q^{-2}},
\;\;\;\;\;
C_k = q^{-2(k-1)}\,\frac{(q^{2k}-1)(q^{2(k-1)}-1)}{(q^{4}-1)(q^{2}-1)}.
\end{align*}
for $k \geq 2$.
\end{enumerate}
\end{lemma}
\begin{proposition}
{\cite[Corollary 2.2]{BeraMukherjee2026UqB2}}\label{centro2} 
The elements \(z\), \(e_{1}^{l}\), \(e_{2}^{l}\), and \(e_{3}^{l}\) are central in \(U_q^{+}(B_2)\).
\end{proposition}
\begin{proof}
It follows that
\begin{align*}
 q^{2l}=1,\quad [l]_{q^{2}}=0,\quad [l]_{q^{-4}}=0,\quad B_l=0,\quad C_l=0.
\end{align*}
From the defining relations of \(U_q^{+}(B_2)\), it follows that $z$ is a central element. On the other hand, by  Lemma \ref{lem:B2}--(i) with $k=l$ one obtains 
\[
e_2 e_3^{\,l}=q^{2l} e_3^{\,l} e_2 + [l]_{q^{2}} z e_3^{\,l-1}=e_3^{\,l} e_2.
\]
Moreover, from the relation $e_1 e_3=q^{-2} e_3 e_1$ it follows by induction that 
$e_1 e_3^{\,l}=q^{-2l} e_3^{\,l} e_1=e_3^{\,l} e_1$. Hence,  $e_3^{\,l}$ commute with $e_1$ and $e_2$.
Analogously, Lemma \ref{lem:B2}--(ii), for $k=l$, it follows
\[
e_2^{\,l} e_3 = q^{2l} e_3 e_2^{\,l} + [l]_{q^{2}} z e_2^{\,l-1}= e_3 e_2^{\,l}.
\]
Once again, Lemma \ref{lem:B2}--(iv) implies
\[
e_2^{\,l} e_1 = q^{-2l} e_1 e_2^{\,l} - q^{-2} B_l\, e_3 e_2^{\,l-1} - C_l\, z e_2^{\,l-2}
= e_1 e_2^{\,l};
\]
that is, $e_2^{\,l}$ commute with $e_1$ and $e_3$. From the relation $e_1 e_3=q^{-2} e_3 e_1$ can be proved that $e_3 e_1^{\,l}=q^{2l} e_1^{\,l} e_3 = e_1^{\,l} e_3$. Indeed, Lemma \ref{lem:B2}--(iii), assuming $k=l$, implies
\[
e_2 e_1^{\,l}= q^{-2l} e_1^{\,l} e_2 - q^{-2}[l]_{q^{-4}} e_3 e_1^{\,l-1}
= e_1^{\,l} e_2.
\]
Thus, $e_1^{\,l}$ commute  with $e_2$ and $e_3$. 
Finally, since  $e_1^{\,l}$, $e_2^{\,l}$, $e_3^{\,l}$ and $z$ commute with the generators of $U_q^{+}(B_2)$, it is concluded that these elements belong to its center.
\end{proof}

\begin{proposition}{\cite[Proposition 2.3]{BeraMukherjee2026UqB2}}\label{Propi} Let $q\in\K^*$ with $q^4\neq1$.  $U_q^{+}(B_2)$ is a PI algebra if and only if $q$ is a root of unity.
\end{proposition}
\begin{proof}
Consider the subalgebra $A:= \K\langle e_1,e_3\rangle$ of $U_q^{+}(B_2)$. The relation  $e_1e_3=q^{-2}e_3e_1$ shows that $A$ is isomorphic to the quantum plane with parameter $\lambda=q^{-2}$:
\[
A\ \cong\  \K\langle x,y\mid yx=\lambda xy\rangle,\qquad \lambda=q^{-2}.
\]
Hence, by \cite[Section \S 7.1]{DeConciniProcesi1993} one obtains  $q^{-2}$ is a root of unity, which implies that $q$ is a root of unity. Conversely, suppose that \(q\) is a root of unity and let
\(\ell=\operatorname{ord}(q^2)\). By
Proposition~\ref{centro2},

$$
R:=\K[e_1^\ell,e_2^\ell,e_3^\ell,z]
\subseteq Z(U_q^{+}(B_2)).
$$

For a monomial \(e_1^ae_2^be_3^c\), write
\(a=u\ell+i\), \(b=v\ell+j\), and \(c=w\ell+k\), with
\(0\leq i,j,k<\ell\). Using the identities in
Lemma~\ref{lem:B2} together with the relation
\(e_1e_3=q^{-2}e_3e_1\), the factors
\(e_1^\ell,e_2^\ell,e_3^\ell\) can be moved successively to the left, so that
$$
e_1^ae_2^be_3^c
\in
\sum_{0\leq i,j,k<\ell}
R\,e_1^ie_2^je_3^k.
$$
Hence, \(U_q^{+}(B_2)\) is finitely generated as an \(R\)-module, with
generating set
$$
\mathcal B=
\{e_1^ie_2^je_3^k:0\leq i,j,k<\ell\}.
$$
Since \(R\) is central, \cite[Corollary~13.1.13]{mcconnell} implies that
\(U_q^{+}(B_2)\) is PI.
\end{proof}
The central elements obtained above allow us to describe the center of \(U_q^{+}(B_2)\) completely. Indeed, Bera and Mukherjee \cite[Theorem~9.2]{BeraMukherjee2026UqB2} proved that it is generated by these elements together with an additional central element \(z_1\).
\begin{lemma}{\cite[Theorem~9.2]{BeraMukherjee2026UqB2}}
\label{centroU}
$$
Z\!\left(U_q^{+}(B_2)\right)
=
\K[e_1^\ell,e_2^\ell,e_3^\ell,z,z_1],
\qquad
z_1
=
e_1e_2e_3
+
\frac{e_1z}{q^2-1}
+
\frac{e_3^2}{q^4-1}.
$$
\end{lemma}
\begin{proposition}
\label{rigidez-ozono-UqB2}
Let \(\varphi\in\operatorname{Aut}(U_q^+(B_2))\) such that
\(\varphi(e_1^\ell)=e_1^\ell\) and \(\varphi(e_2^\ell)=e_2^\ell\). Then
there exist \(a,b\in\mu_\ell\) such that
$$
\varphi(e_1)=ae_1,\qquad
\varphi(e_2)=be_2,\qquad
\varphi(e_3)=ab\,e_3,\qquad
\varphi(z)=ab^2z.
$$
\end{proposition}
\begin{proof}
Let \(A=U_q^+(B_2)\). Consider the standard grading on \(A\), determined by

$$
\deg(e_1)=\deg(e_2)=1,\qquad
\deg(e_3)=2,\qquad
\deg(z)=3.
$$
The defining relations are homogeneous, so \(A\) is a connected
\(\mathbb N\)-graded algebra. Moreover, \(A\) admits a presentation as an
iterated Ore extension over \(\K[z,e_3]\), and hence it is a domain.
Consequently, \(\deg(x^\ell)=\ell\deg(x)\) for every \(x\neq0\). From
\(\varphi(e_1^\ell)=e_1^\ell\), we obtain
$$
\ell\deg(\varphi(e_1))
=\deg(\varphi(e_1)^\ell)
=\deg(e_1^\ell)
=\ell,
$$
and therefore \(\deg(\varphi(e_1))=1\). Since
\(A_0=\K\) and \(A_1=\K e_1\oplus\K e_2\), there exist \(a,b,c\in\K\) such that
$$
\varphi(e_1)=ae_1+be_2+c.
$$
Comparing the degree-zero components in
\(\varphi(e_1)^\ell=e_1^\ell\), we obtain \(c^\ell=0\), and hence \(c=0\). Thus,
$$
\varphi(e_1)=ae_1+be_2.
$$
Now consider the quotient \(A/\langle e_1\rangle\). Since \(e_1=0\), the relations

$$
e_2e_1=q^{-2}e_1e_2-q^{-2}e_3,
\qquad
e_2e_3=q^2e_3e_2+z
$$
imply that \(e_3=z=0\). Hence, \(A/\langle e_1\rangle\cong\K[e_2]\). Reducing the equality
$$
(ae_1+be_2)^\ell=e_1^\ell
$$
modulo \(\langle e_1\rangle\), we obtain \(b^\ell e_2^\ell=0\). Since
\(\K[e_2]\) is a domain, \(b=0\), and therefore \(\varphi(e_1)=ae_1\).
Moreover, from \(\varphi(e_1^\ell)=e_1^\ell\), we obtain \(a^\ell=1\), that is,
\(a\in\mu_\ell\).
\noindent Analogously, from \(\varphi(e_2^\ell)=e_2^\ell\) and the quotient
\(A/\langle e_2\rangle\cong\K[e_1]\), we obtain
$$
\varphi(e_2)=be_2
$$
for some \(b\in\mu_\ell\). Finally,
$$
\varphi(e_3)=ab(e_1e_2-q^2e_2e_1)=ab\,e_3
\qquad\text{and}\qquad
\varphi(z)=ab^2(e_2e_3-q^2e_3e_2)=ab^2z.
$$
\end{proof}
\begin{theorem}\label{ozono-UqB2}
$$
\operatorname{Oz}(U_q^+(B_2))\cong \mu_{\gcd(\ell,2)}.
$$
More precisely,
$$
\operatorname{Oz}(U_q^+(B_2))
\cong
\begin{cases}
\{1\},&\ell\text{ odd},\\[1mm]
\mu_2,&\ell\text{ even}.
\end{cases}
$$
\end{theorem}

\begin{proof}
Let \(\varphi\in\operatorname{Oz}(A)\). Since
\(e_1^\ell,e_2^\ell\in Z(A)\), by Proposition~\ref{rigidez-ozono-UqB2}
there exist \(a,b\in\mu_\ell\) such that
$$
\varphi(e_1)=ae_1,\qquad
\varphi(e_2)=be_2,\qquad
\varphi(e_3)=ab\,e_3,\qquad
\varphi(z)=ab^2z.
$$
Since \(z\in Z(U_q^{+}(B_2))\), we have \(\varphi(z)=z\), and hence
\(ab^2=1\). Likewise, since \(z_1\in Z(U_q^{+}(B_2))\),
$$
\varphi(z_1)
=a^2b^2\left(e_1e_2e_3+\frac{e_1z}{q^2-1}+\frac{e_3^2}{q^4-1}\right)
=a^2b^2z_1=z_1.
$$
Thus, \(a^2b^2=1\). Together with \(ab^2=1\), this implies \(a=1\) and
\(b^2=1\). Moreover, \(b^\ell=1\); therefore,
\(b\in\mu_{\gcd(\ell,2)}\). Consequently,
$$
\operatorname{Oz}(U_q^{+}(B_2))\subseteq
\{\psi_\varepsilon:\varepsilon\in\mu_{\gcd(\ell,2)}\}.
$$
Conversely, let \(\varepsilon\in\mu_{\gcd(\ell,2)}\). Then
\(\varepsilon^\ell=\varepsilon^2=1\). Consider the graded automorphism
$$
\psi_\varepsilon(e_1)=e_1,\qquad
\psi_\varepsilon(e_2)=\varepsilon e_2.
$$
Then \(\psi_\varepsilon(e_3)=\varepsilon e_3\) and
\(\psi_\varepsilon(z)=z\). Moreover,
$$
\psi_\varepsilon(e_1^\ell)=e_1^\ell,\qquad
\psi_\varepsilon(e_2^\ell)=e_2^\ell,\qquad
\psi_\varepsilon(e_3^\ell)=e_3^\ell,\qquad
\psi_\varepsilon(z_1)=\varepsilon^2z_1=z_1.
$$
Since
\(Z(U_q^{+}(B_2))=\K[e_1^\ell,e_2^\ell,e_3^\ell,z,z_1]\),
\(\psi_\varepsilon\) fixes \(Z(U_q^{+}(B_2))\) pointwise. Therefore,
$$
\psi_\varepsilon\in\operatorname{Oz}(U_q^{+}(B_2)).
$$
\end{proof}

\begin{proposition}\label{prop:normal-explicito}
Suppose that \(\ell\) is even. Then \(e_1^{\ell/2}\) is a noncentral normal
element of \(U_q^{+}(B_2)\), and
\[
e_1^{\ell/2}\,e_1=e_1\,e_1^{\ell/2},\qquad
e_1^{\ell/2}\,e_2=-\,e_2\,e_1^{\ell/2},\qquad
e_1^{\ell/2}\,e_3=-\,e_3\,e_1^{\ell/2},\qquad
e_1^{\ell/2}\,z=z\,e_1^{\ell/2}.
\]
Consequently \(\eta_{e_1^{\ell/2}}=\psi_{-1}\) generates
\(\operatorname{Oz}(U_q^{+}(B_2))\).
\end{proposition}

\begin{proof}
Put \(k=\ell/2\). From \(e_1e_3=q^{-2}e_3e_1\) one gets by induction
\(e_1^{k}e_3=q^{-2k}e_3e_1^{k}\); since \(q^{2}\) has order \(\ell\) and
\(0<k<\ell\), the scalar \(q^{-2k}=(q^{2})^{-\ell/2}\) is a square root of
\(1\) different from \(1\), that is, \(q^{-2k}=-1\). Hence
\(e_1^{k}e_3=-e_3e_1^{k}\).

By Lemma~\ref{lem:B2}(iii),
\[
e_2 e_1^{k}=q^{-2k} e_1^{k} e_2-q^{-2}\,[k]_{q^{-4}}\, e_3 e_1^{k-1},
\qquad
[k]_{q^{-4}}=\frac{q^{-4k}-1}{q^{-4}-1}=\frac{q^{-2\ell}-1}{q^{-4}-1}=0,
\]
because \(q^{2\ell}=1\) and \(q^{4}\neq1\). Therefore
\(e_2e_1^{k}=-e_1^{k}e_2\). Since \(z\) is central, \(e_1^{k}\) commutes with
\(z\) and with \(e_1\), so \(e_1^{k}U_q^{+}(B_2)=U_q^{+}(B_2)e_1^{k}\) and
\(e_1^{k}\) is normal; it is not central, since $e_1^{k}e_2=-e_2e_1^{k}$ and $e_2e_1^{k}\neq0$ in the
domain $U_q^{+}(B_2)$. Finally, with the convention
\(\eta_a(x)=a^{-1}xa\) used in \cite[Lemma~1.10]{ChanGaddisWonZhang2025}, the
relations above give \(\eta_{e_1^{k}}(e_1)=e_1\),
\(\eta_{e_1^{k}}(e_2)=-e_2\), \(\eta_{e_1^{k}}(e_3)=-e_3\) and
\(\eta_{e_1^{k}}(z)=z\), that is, \(\eta_{e_1^{k}}=\psi_{-1}\).
\end{proof}

\begin{corollary}
\label{cor:rango-ozono-UqB2}
\begin{align*}
\operatorname{rk}_{Z(U_q^{+}(B_2))}\bigl(U_q^{+}(B_2)\bigr)=\ell^2
\qquad\text{and}\qquad
\bigl|\operatorname{Oz}\bigl(U_q^{+}(B_2)\bigr)\bigr|
<\operatorname{rk}_{Z(U_q^{+}(B_2))}\bigl(U_q^{+}(B_2)\bigr).
\end{align*}
\end{corollary}

\begin{proof}
Write $A=U_q^{+}(B_2)$ and $Z=Z(A)$. Since $A$ is a domain, it is prime; it is
PI and finitely generated as a module over $Z$, and
$\operatorname{PIdeg}(A)=\ell$ by \cite[Theorem~2.7]{BeraMukherjee2026UqB2}.
Let $Q$ denote the total ring of fractions of $A$. By Posner's theorem
\cite[Theorem~13.6.5]{mcconnell}, $Q=A(Z\setminus\{0\})^{-1}$ is a central
simple algebra of degree $\operatorname{PIdeg}(A)$ over its center. This
center is the field of fractions $Q(Z)$; this can also be seen from
\cite[Theorem~2.1]{LezamaVenegas2020Serdica}, whose hypothesis
$\GKdim(A)<\GKdim(Z)+1$ holds because $\GKdim(A)=\GKdim(Z)$ for an algebra
that is finite over its center. Hence
\begin{align*}
\operatorname{rk}_{Z}(A)
=\dim_{Q(Z)}Q
=\operatorname{PIdeg}(A)^2
=\ell^2.
\end{align*}
By Theorem~\ref{ozono-UqB2}, $|\operatorname{Oz}(A)|=\gcd(\ell,2)\leq 2$.
Since $m\geq5$, we have $\ell\geq3$, and so
\begin{align*}
|\operatorname{Oz}(A)|=\gcd(\ell,2)\leq 2<9\leq\ell^2=\operatorname{rk}_{Z}(A).
\end{align*}
\end{proof}

\section{Some homological properties of
\texorpdfstring{$U_q^{+}(B_2)$}{Uq+(B2)}}
Graded skew PBW extensions were introduced by Suárez
\cite{Suarez2017KoszulityGradedSkewPBW} in 2017 as a generalization of
graded iterated Ore extensions. Subsequently, the first author and Suárez
\cite{GomezSuarez2020} studied their relationship with double Ore
extensions and proved that every graded skew PBW extension in two
variables over an Artin--Schelter regular algebra is again
Artin--Schelter regular. In what follows, we show that
\(U_q^{+}(B_2)\) can be realized as a graded skew PBW extension in two
variables and, as a consequence, establish that \(U_q^{+}(B_2)\) is
Artin--Schelter regular. We also show that it is Auslander-regular, Cohen--Macaulay, strongly
Noetherian and skew Calabi--Yau, and in Section~\ref{sec:CY} that it is in
fact Calabi--Yau.

\begin{definition}[{\cite[Definition~1]{GallegoLezama}}]
Let \(R\subseteq A\) be rings. We say that \(A\) is a
\emph{skew PBW extension} of \(R\), and write
$$
A=\sigma(R)\langle x_1,\ldots,x_n\rangle,
$$
if \(A\) is a free left \(R\)-module with basis
$$
\operatorname{Mon}(A)=
\{x_1^{\alpha_1}\cdots x_n^{\alpha_n}:
(\alpha_1,\ldots,\alpha_n)\in\mathbb N^n\},
$$
and there exist injective endomorphisms \(\sigma_i:R\to R\),
\(\sigma_i\)-derivations \(\delta_i:R\to R\), and elements
\(c_{i,j}\in R\setminus\{0\}\) such that
$$
x_ir=\sigma_i(r)x_i+\delta_i(r),\qquad
x_jx_i-c_{i,j}x_ix_j\in R+Rx_1+\cdots+Rx_n,
$$
for \(r\in R\) and \(1\leq i<j\leq n\).
The extension is called \emph{bijective} if each \(\sigma_i\) is an
automorphism and each \(c_{i,j}\) is invertible in \(R\).
\end{definition}

\begin{definition}[{\cite[Proposition~2.7]{Suarez2017KoszulityGradedSkewPBW}}]
Let \(R=\bigoplus_{m\geq0}R_m\) be an \(\mathbb N\)-graded algebra and
let \(A=\sigma(R)\langle x_1,\ldots,x_n\rangle\) be a bijective skew PBW
extension of \(R\). We say that \(A\) is a \emph{graded skew PBW
extension} if \(\deg(x_i)=1\), each \(\sigma_i\) preserves the grading,
\(\delta_i(R_m)\subseteq R_{m+1}\), and, for \(1\leq i<j\leq n\),
$$
x_jx_i-c_{i,j}x_ix_j\in R_2+R_1x_1+\cdots+R_1x_n,
\qquad c_{i,j}\in R_0.
$$
In this case,
$$
A_p=\operatorname{span}_{\K}
\{r_tx_1^{\alpha_1}\cdots x_n^{\alpha_n}:
r_t\in R_t,\ t+\alpha_1+\cdots+\alpha_n=p\}.
$$
\end{definition}
\begin{proposition}
\label{prop:PBW-graduada-UqB2}
Let \(q\in\K^*\). Then \(U_q^{+}(B_2)\) is a graded bijective skew PBW extension
in two variables over \(R=\K[z,e_3]\), where \(\deg(e_3)=2\),
\(\deg(z)=3\), and \(\deg(e_1)=\deg(e_2)=1\).
\end{proposition}
\begin{proof}
Let \(R=\K[z,e_3]\), endowed with the grading determined by
\(\deg(e_3)=2\) and \(\deg(z)=3\). By
\cite[Section~3.1.2]{AndruskiewitschDumas2008}, the monomials
$$
\{z^ie_3^je_1^ke_2^r:i,j,k,r\geq0\}
$$
form a PBW basis of \(U_q^{+}(B_2)\). Hence,
\(U_q^{+}(B_2)\) is a free left \(R\)-module with basis
\(\{e_1^ke_2^r:k,r\geq0\}\).
The relations of \(e_1,e_2\) with \(R\) are determined by
\(\sigma_1(z)=z\), \(\sigma_1(e_3)=q^{-2}e_3\),
\(\sigma_2(z)=z\), \(\sigma_2(e_3)=q^2e_3\),
\(\delta_1=0\), \(\delta_2(z)=0\), and \(\delta_2(e_3)=z\), together with
$$
e_2e_1-q^{-2}e_1e_2=-q^{-2}e_3\in R_2.
$$
Thus, \(U_q^{+}(B_2)\) is a skew PBW extension of \(R\).
Moreover, \(\sigma_1,\sigma_2\) are automorphisms and
\(q^{-2}\in R_0^*\), so the extension is bijective. Finally,
\(\deg(e_1)=\deg(e_2)=1\), the maps \(\sigma_1,\sigma_2\) preserve the
grading, \(\delta_1=0\) is compatible with the grading, and \(\delta_2\)
has degree \(1\), since
\(\deg(\delta_2(e_3))=3=\deg(e_3)+1\) and \(\delta_2(z)=0\).
Furthermore, the relation above satisfies the graded condition because
its remaining term belongs to \(R_2\) and \(q^{-2}\in R_0\). Therefore,
\(U_q^{+}(B_2)=\sigma(R)\langle e_1,e_2\rangle\) is a graded skew PBW
extension of \(R\).
\end{proof}
\begin{theorem}
\label{AS-regular-UqB2}
Let \(q\in\K^*\). Then \(U_q^{+}(B_2)\) is an
Artin--Schelter regular algebra of global dimension \(4\).
\end{theorem}
\begin{proof}
By Proposition~\ref{prop:PBW-graduada-UqB2},
\(U_q^{+}(B_2)\) is a graded skew PBW extension in two variables
over \(R=\K[z,e_3]\), where \(\deg(e_3)=2\) and \(\deg(z)=3\).
Since \(R\) is a commutative graded polynomial ring in two homogeneous
variables of positive degrees, it is Artin--Schelter regular of global
dimension \(2\). By \cite[Theorem~3.10]{GomezSuarez2020},
\(U_q^{+}(B_2)\) is Artin--Schelter regular and
$$\operatorname{gldim}(U_q^{+}(B_2))=\operatorname{gldim}(R)+2=4.$$
\end{proof}

\begin{remark}
As noted in the Introduction, \(U_q^{+}(B_2)\) admits a presentation as
a graded iterated Ore extension. Since the corresponding automorphisms
preserve the grading and the derivation is homogeneous of degree \(1\),
successive applications of
\cite[Proposition~2]{ArtinSchelterTate1991Quantum} also show that
\(U_q^{+}(B_2)\) is Artin--Schelter regular of global dimension \(4\).

\noindent Thus, the iterated Ore extension presentation provides an alternative
and direct proof of the previous theorem. However, realizing
\(U_q^{+}(B_2)\) as a bijective skew PBW extension allows one to apply
systematically a variety of transfer results from the base algebra to the
extension. In particular, the theory of skew PBW extensions provides
general criteria for preserving properties such as Auslander regularity,
the Cohen--Macaulay property, and strong Noetherianity
\cite{LezamaVenegas2017Homological}, as well as other homological,
structural, and \(K\)-theoretic results for bijective skew PBW extensions
\cite{LezamaReyes2014Homological}.
\end{remark}

\begin{theorem}
\label{teo:UqB2-Auslander-CM}
Let \(q\in\K^*\). Then \(U_q^{+}(B_2)\) is Auslander--regular, Cohen--Macaulay, and strongly Noetherian.
\end{theorem}

\begin{proof}
The algebra \(R=\K[z,e_3]\) is a commutative polynomial ring; in
particular, it is Auslander--regular, Cohen--Macaulay, and strongly
Noetherian. By Proposition~\ref{prop:PBW-graduada-UqB2},
\(U_q^{+}(B_2)=\sigma(R)\langle e_1,e_2\rangle\) is a bijective skew PBW
extension of \(R\). Hence, by
\cite[Theorem~2.9]{LezamaVenegas2017Homological},
\(U_q^{+}(B_2)\) is Auslander--regular.

Moreover, \(R\) is connected graded and
\(\sigma_i(R_n)\subseteq R_n\) for every \(n\geq0\) and \(i=1,2\).
Therefore, \cite[Theorem~3.9]{LezamaVenegas2017Homological} implies that
\(U_q^{+}(B_2)\) is Cohen--Macaulay. Finally, since \(R\) is strongly
Noetherian, \cite[Theorem~4.3]{LezamaVenegas2017Homological} implies that
\(U_q^{+}(B_2)\) is strongly Noetherian.
\end{proof}

Reyes, Rogalski, and
Zhang also proved that a connected graded algebra is skew
Calabi--Yau if and only if it is Artin--Schelter regular
\cite[Lemma~1.2]{ReyesRogalskiZhang2014}.
\begin{theorem}
\label{teo:UqB2-skew-CY}
Let \(q\in\K^*\). Then \(U_q^{+}(B_2)\) is skew Calabi--Yau
of dimension \(4\).
\end{theorem}
\begin{proof}
By Theorem~\ref{AS-regular-UqB2},
\(U_q^{+}(B_2)\) is Artin--Schelter regular of global dimension \(4\).
Moreover, with the grading considered above,
\(U_q^{+}(B_2)\) is a connected graded algebra. Therefore, by
\cite[Lemma~1.2]{ReyesRogalskiZhang2014},
\(U_q^{+}(B_2)\) is skew Calabi--Yau of dimension \(4\).
\end{proof}

\subsection{Some applications of the ozone group computation}
Throughout this subsection, let $q\in\K^*$ be a primitive
$m$-th root of unity with $m\geq5$, and set
\[
A=U_q^{+}(B_2),
\qquad
\ell=\operatorname{ord}(q^2)=\frac{m}{\gcd(m,2)}.
\]
The preceding results show that $A$ is PI, Artin--Schelter
regular, Auslander--regular, and strongly Noetherian, hence
Noetherian. Moreover,
\[
\operatorname{Oz}(A)\cong\mu_{\gcd(\ell,2)}.
\]
These properties allow us to apply the general results of
Chan, Gaddis, Won, and Zhang \cite{ChanGaddisWonZhang2025}.

\begin{remark}\label{teo:UqB2-center-Gorenstein}
If $\ell$ is odd, then $\operatorname{Oz}(A)=\{1\}$.
Consequently, \cite[Theorem~1.11]{ChanGaddisWonZhang2025}
implies that $A$ is Calabi--Yau and $Z(A)$ is Gorenstein
in the usual commutative sense. The next subsection establishes
the Calabi--Yau property for every $q\in\K^*$.
\end{remark}

We now describe the nonzero normal elements of $A$ modulo
multiplication by nonzero central elements.

\begin{proposition}
\label{prop:normales-ozono-UqB2}
Let \(N^*\) be the multiplicative monoid of nonzero normal elements of
\(U_q^{+}(B_2)\), and let \(\sim\) be the equivalence relation defined by
$$
a\sim b
\quad\Longleftrightarrow\quad
z_1a=z_2b
$$
for some
\(z_1,z_2\in Z(U_q^{+}(B_2))\setminus\{0\}\). Then
$$
N^*/{\sim}
\cong
\operatorname{Oz}(U_q^{+}(B_2))
\cong
\mu_{\gcd(\ell,2)}.
$$
In particular, if \(\ell\) is odd, every normal element of
\(U_q^{+}(B_2)\) is central; if \(\ell\) is even, the class of \(e_1^{\ell/2}\) generates
\(N^*/{\sim}\).
\end{proposition}
\begin{proof}
In the proof of \cite[Lemma~1.10]{ChanGaddisWonZhang2025}, the authors
consider the surjective map
$$
N^*\longrightarrow\operatorname{Oz}(U_q^{+}(B_2)),
\qquad
a\longmapsto\eta_a,\quad \eta_a(x)=a^{-1}xa,
$$
whose surjectivity follows from
\cite[Lemma~1.9]{ChanGaddisWonZhang2025}. Moreover, for the equivalence
relation \(a\sim b\) whenever \(z_1a=z_2b\) for some
\(z_1,z_2\in Z(U_q^{+}(B_2))\setminus\{0\}\), one has
$$
N^*/{\sim}\cong\operatorname{Oz}(U_q^{+}(B_2)).
$$
By Theorem~\ref{ozono-UqB2},
$$
N^*/{\sim}\cong\operatorname{Oz}(U_q^{+}(B_2))
\cong\mu_{\gcd(\ell,2)}.
$$
If \(\ell\) is odd, \cite[Lemma~1.10]{ChanGaddisWonZhang2025}
implies that every normal element of \(U_q^{+}(B_2)\) is central; If \(\ell\) is even, Proposition
\ref{prop:normal-explicito} shows that \(e_1^{\ell/2}\) is normal and
noncentral with \(\eta_{e_1^{\ell/2}}=\psi_{-1}\), which generates
\(\operatorname{Oz}(U_q^{+}(B_2))\); hence its class generates
\(N^*/{\sim}\).
\end{proof}

The ozone group computation also yields homological information about the
corresponding invariant ring.

\begin{proposition}\label{prop:UqB2-hdet-invariants}
Let $q\in\K^*$ be a primitive $m$-th root of unity with $m\geq5$, and set
$A=U_q^{+}(B_2)$ and $G=\operatorname{Oz}(A)$. Then the action of $G$ on $A$ has
trivial homological determinant. Consequently, $A^{G}$ is Artin--Schelter Gorenstein.
\end{proposition}

\begin{proof}
Since the defining relations of $A$ are homogeneous with respect to the
$\mathbb{Z}^2$-grading given by $e_1\mapsto(1,0)$ and $e_2\mapsto(0,1)$, for every
$\alpha,\beta\in\K^*$ there is a graded automorphism $\psi_{\alpha,\beta}$ of $A$ with
\[
\psi_{\alpha,\beta}(e_1)=\alpha e_1,\qquad
\psi_{\alpha,\beta}(e_2)=\beta e_2,\qquad
\psi_{\alpha,\beta}(e_3)=\alpha\beta e_3,\qquad
\psi_{\alpha,\beta}(z)=\alpha\beta^{2}z .
\]
The PBW basis consists of eigenvectors of $\psi_{\alpha,\beta}$, so
\[
\operatorname{Tr}_A(\psi_{\alpha,\beta},t)
=\frac{1}{(1-\alpha t)(1-\beta t)(1-\alpha\beta t^{2})(1-\alpha\beta^{2}t^{3})}
=\frac{1}{\alpha^{3}\beta^{4}}\,t^{-7}+O(t^{-8})
\]
as a Laurent series in $t^{-1}$. Since $A$ is Artin--Schelter regular of dimension
$4$, \cite[Lemma~2.6]{JorgensenZhang2000Gorensteinness} gives
$\operatorname{Tr}_A(\psi_{\alpha,\beta},t)
=(-1)^{4}\operatorname{hdet}(\psi_{\alpha,\beta})^{-1}t^{-7}+O(t^{-8})$, and hence
\[
\operatorname{hdet}(\psi_{\alpha,\beta})=\alpha^{3}\beta^{4}.
\]
By Theorem~3.7, $G$ is either trivial or generated by $\psi_{1,-1}$, and
$\operatorname{hdet}(\psi_{1,-1})=1$. Thus $G$ acts with trivial homological
determinant. Since $A$ is Noetherian and Artin--Schelter regular and
$\operatorname{char}(\K)=0$, \cite[Theorem~3.3]{JorgensenZhang2000Gorensteinness}
implies that $A^{G}$ is Artin--Schelter Gorenstein.
\end{proof}

\begin{remark}\label{rem:UqB2-hdet-CY}
The triviality of the homological determinant in
Proposition~\ref{prop:UqB2-hdet-invariants} also follows from
\cite[Theorem~3.8]{LiuWuZhu2026}, in view of
Theorem~\ref{thm:Nakayama-CY-UqB2}.
The computation above establishes this directly, without
using the Calabi--Yau property.
\end{remark}

\begin{remark}
The statement is nontrivial only when $4\mid m$. In that case
$G=\langle\psi_{1,-1}\rangle$, and $A^{G}$ has a basis consisting of the PBW monomials
$z^{i}e_3^{j}e_1^{k}e_2^{r}$ with $j+r$ even. In particular,
$Z(A)\subsetneq A^{G}\subsetneq A$, since $e_1\in A^{G}\setminus Z(A)$ and
$e_2\notin A^{G}$. Moreover, $\operatorname{rk}_{A^{G}}(A)=|G|=2$ by
\cite[Lemma~3.4(5)]{ChanGaddisWonZhang2025}.
\end{remark}

\subsection{The Calabi--Yau property of \texorpdfstring{$U_q^{+}(B_2)$}{Uq+(B2)}}\label{sec:CY}

We now show that the Nakayama automorphism of $U_q^{+}(B_2)$ is trivial for every $q\in\K^*$. We use the following result on Ore extensions: if $B$ is skew Calabi--Yau
(called twisted Calabi--Yau in \cite{LiuWangWu2014Ore}) of dimension $d$ with Nakayama
automorphism $\mu_B$, and $E=B[x;\sigma,\delta]$ with $\sigma$ an automorphism, then
by \cite[Theorem~3.3]{LiuWangWu2014Ore} the algebra $E$ is skew Calabi--Yau of
dimension $d+1$ and admits a Nakayama automorphism $\mu_E$ satisfying
\[
\mu_E|_B=\sigma^{-1}\circ\mu_B .
\]
Since the original presentation of $U_q^{+}(B_2)$ only allows us to compute
$\mu_{U_q^{+}(B_2)}(e_1)$, we first give a second presentation, in which $e_1$ is the
last variable.

The use of a second presentation for this purpose was suggested to us by
Xin Tang.
\begin{lemma}\label{lem:segunda-presentacion-Ore}
Let $q\in\K^*$. Then
\[
U_q^{+}(B_2)=\K[e_2,z][e_3;\tau_3,\partial_3][e_1;\tau_1,\partial_1],
\]
where
\[
\tau_3(e_2)=q^{-2}e_2,\qquad \tau_3(z)=z,\qquad
\partial_3(e_2)=-q^{-2}z,\qquad \partial_3(z)=0,
\]
and
\[
\tau_1(e_2)=q^{2}e_2,\qquad \tau_1(e_3)=q^{-2}e_3,\qquad \tau_1(z)=z,\qquad
\partial_1(e_2)=e_3,\qquad \partial_1(e_3)=\partial_1(z)=0 .
\]
\end{lemma}

\begin{proof}
The map $\tau_3$ is an automorphism of $\K[e_2,z]$, and the
prescribed values of $\partial_3$ define a $\tau_3$-derivation,
since
\[
\partial_3(e_2z-ze_2)=-q^{-2}z^2+q^{-2}z^2=0.
\]
Hence
\[
C_1=\K[e_2,z][e_3;\tau_3,\partial_3]
\]
is well defined.

Set $r=e_3e_2-q^{-2}e_2e_3+q^{-2}z$.
The diagonal map $\tau_1$ and its inverse preserve the defining
relations of $C_1$, so $\tau_1$ is an automorphism. Moreover,
\[
\begin{aligned}
\partial_1(r)&=q^{-2}e_3^2-q^{-2}e_3^2=0,\\
\partial_1(e_2z-ze_2)&=e_3z-ze_3=0,\\
\partial_1(e_3z-ze_3)&=0.
\end{aligned}
\]
Thus $\partial_1$ defines a $\tau_1$-derivation of $C_1$,
and $E=C_1[e_1;\tau_1,\partial_1]$ is well defined.

By the universal property of Ore extensions, $E$ is the algebra
generated by $e_1,e_2,e_3,z$ subject to
\[
e_3e_2=q^{-2}e_2e_3-q^{-2}z,\qquad
e_1e_2=q^2e_2e_1+e_3,\qquad
e_1e_3=q^{-2}e_3e_1,
\]
with $z$ central. These relations are equivalent to the defining
relations of $U_q^{+}(B_2)$, and therefore $E=U_q^{+}(B_2)$.
\end{proof}

\begin{theorem}\label{thm:Nakayama-CY-UqB2}
Let $q\in\K^*$. The Nakayama automorphism of $U_q^{+}(B_2)$ is the identity.
Consequently, $U_q^{+}(B_2)$ is Calabi--Yau of dimension $4$.
\end{theorem}

\begin{proof}
Let $A=U_q^{+}(B_2)$. By Theorem~\ref{teo:UqB2-skew-CY}, $A$ is skew Calabi--Yau of
dimension $4$. Since $A$ is a connected graded domain, its units are the elements of
$\K^*$, so its inner automorphisms are trivial. As the Nakayama automorphism is unique
up to inner automorphisms, $\mu_A$ is uniquely determined, and we may compute it using
different Ore presentations.

The subalgebra generated by $z,e_3,e_1$ is $\K[z,e_3][e_1;\sigma_1]$. Since $z$ is
central and $e_3e_1=q^{2}e_1e_3$, this subalgebra is also
\[
B_1=B_0[e_3;\rho_3],\qquad B_0=\K[e_1,z],\qquad
\rho_3(e_1)=q^{2}e_1,\quad \rho_3(z)=z,
\]
and therefore $A=B_1[e_2;\sigma_2,\delta_2]$, with $\sigma_2(e_1)=q^{-2}e_1$. Since
$\mu_{B_0}=\operatorname{id}$, two applications of
\cite[Theorem~3.3]{LiuWangWu2014Ore} give
\[
\mu_{B_1}(e_1)=\rho_3^{-1}(e_1)=q^{-2}e_1,
\qquad
\mu_A(e_1)=\sigma_2^{-1}(q^{-2}e_1)=e_1 .
\]

Similarly, by Lemma~\ref{lem:segunda-presentacion-Ore},
$A=C_1[e_1;\tau_1,\partial_1]$ with $C_1=C_0[e_3;\tau_3,\partial_3]$ and
$C_0=\K[e_2,z]$. Since $\mu_{C_0}=\operatorname{id}$, we obtain
\[
\mu_{C_1}(e_2)=\tau_3^{-1}(e_2)=q^{2}e_2,
\qquad
\mu_A(e_2)=\tau_1^{-1}(q^{2}e_2)=e_2 .
\]
Since $A$ is generated by $e_1$ and $e_2$, we conclude that $\mu_A=\operatorname{id}_A$.
\end{proof}
\begin{remark}\label{rem:GY}
The argument above is the method of \cite[Section~4]{LiuWangWu2014Ore},
which Goodearl and Yakimov formalized for reversible diagonalized iterated
Ore extensions \cite[Theorem~3.6]{GoodearlYakimov2015}; their formula does
not require $q$ to be generic. In fact, with the order $e_1,e_3,z,e_2$ the
algebra $U_q^{+}(B_2)$ is such an extension, and
\cite[Theorem~3.6]{GoodearlYakimov2015} gives directly
\begin{align*}
\mu(e_1)&=q^{-2}\cdot 1\cdot q^{2}\,e_1=e_1,\\
\mu(e_2)&=q^{-2}\cdot q^{2}\cdot 1\,e_2=e_2.
\end{align*}
For $q$ not a root of unity, Theorem~\ref{thm:Nakayama-CY-UqB2} is also a
special case of \cite[Example~3.10]{GoodearlYakimov2015} with $w=w_0$; see
also \cite{LiuWu2014}.
\end{remark}

In the following corollary, Gorenstein is understood in the usual commutative sense.

\begin{corollary}\label{cor:center-gorenstein-UqB2}
Let $q\in\K^*$ be a primitive $m$-th root of unity with $m\geq5$. Then
$Z(U_q^{+}(B_2))$ is Gorenstein.
\end{corollary}

\begin{proof}
By Proposition~\ref{Propi}, Theorem~\ref{AS-regular-UqB2} and
Theorem~\ref{teo:UqB2-Auslander-CM}, $U_q^{+}(B_2)$ is a Noetherian PI Artin--Schelter
regular algebra. By Theorem~\ref{thm:Nakayama-CY-UqB2}, it is Calabi--Yau. Hence the
implication $(3)\Rightarrow(4)$ of \cite[Theorem~1.11]{ChanGaddisWonZhang2025}, which
relies on \cite[Lemma~2.5(5)]{StaffordVandenBergh2008}, shows that
$Z(U_q^{+}(B_2))$ is Gorenstein.
\end{proof}

\begin{remark}
If $4\mid m$, then $\operatorname{Oz}(U_q^{+}(B_2))\cong\mu_2$ is generated
by $\eta_{e_1^{\ell/2}}$, while $U_q^{+}(B_2)$ is Calabi--Yau. So, like the
Calabi--Yau skew polynomial rings and the down-up algebra $A(0,-1)$ in
\cite[Lemma~2.2 and Example~4.8(2)]{ChanGaddisWonZhang2025}, it shows that
the implication $(1)\Rightarrow(3)$ in
\cite[Theorem~1.11]{ChanGaddisWonZhang2025} cannot be reversed.
\end{remark}

\begin{question*}
A closely related generalization of \(U_q^{+}(B_2)\) is the algebra
\(U_{r,s}^{+}(B_2)\). When \(r\) and \(s\) are roots of unity, its PI
property, PI degree, and simple modules have recently been studied
\cite{MukherjeePandey2025}, and distinguished normal elements also arise
in this analysis. These features suggest studying its ozone group.
\end{question*}

\subsection*{Future work}
A natural next step is to compute ozone groups for other families of
quantum algebras at roots of unity. Bera, Mukherjee, Mandal and their
collaborators have studied several such families: the two-parameter quantum
Heisenberg algebras $\mathcal H_{p,q}$, quantized Weyl algebras, the
bialgebra $M(p,q)$, and the Dipper--Donkin and reflection equation quantum
matrix algebras
\cite{BeraMandalNandy2024,BeraMukherjee2024DipperDonkin,Mukherjee2026Bialgebra,MukherjeeBera2025}.
For each of them they look at the PI property, the PI degree, central and
normal elements, and simple modules, which is more or less the information
one needs to start computing the ozone group. We will treat
$\mathcal H_{p,q}$ in forthcoming work. For the other families the
question seems to be open.

\section*{Acknowledgments}

We thank the members of the Research Seminar on Noncommutative Algebra
for the discussions, comments, and suggestions that contributed to the
development of this work. We are especially grateful to Professor
Héctor Julio Suárez Suárez for his active participation, guidance, and
continued support, and to Professor Xin Tang for pointing out that two
different iterated Ore presentations of $U_q^{+}(B_2)$ allow one to
determine the action of its Nakayama automorphism on $e_1$ and $e_2$,
an observation that led to the proof of the Calabi--Yau property
presented in this paper.

\bibliographystyle{plain}
\bibliography{biblio}
\bigskip
\noindent
James Yair Gómez Lozano.:\\
Grupo de Álgebra y Análisis,\\
Universidad Pedagógica y Tecnológica de Colombia,\\
Tunja, Colombia\\
james.gomez@uptc.edu.co\\

\noindent Helbert Javier Venegas Ramírez.:\\
Universidad Militar Nueva Granada, Bogot\'a, Colombia,\\
Bogotá, Colombia\\
Helbert.venegas@unimilitar.edu.co

\end{document}